\documentclass[11pt,a4paper]{amsart} 
\pdfoutput=1
\usepackage[hidelinks,pdfstartview=FitH]{hyperref} 
\usepackage{amssymb}
\usepackage[USenglish]{babel}
\usepackage{palatino,eulervm}
\usepackage[T1]{fontenc}
\usepackage[top=3cm,right=3cm,left=3cm,bottom=3cm,marginparsep=15pt,marginparwidth=2cm,footskip=20pt]{geometry}

\allowdisplaybreaks[4]
\numberwithin{equation}{section}

\newtheorem{thm}{Theorem}
\newtheorem{lemma}[thm]{Lemma}
\newtheorem{prop}[thm]{Proposition}
\newtheorem{rem}[thm]{Remark}
\newtheorem{cor}[thm]{Corollary}
\newtheorem{df}[thm]{Definition}

\newtheorem*{nnthm}{Theorem}

\newcommand{\Z}{\mathbb{Z}}
\newcommand{\R}{\mathbb{R}}
\newcommand{\C}{\mathbb{C}}

\newcommand{\id}{\mathrm{Id}}
\newcommand{\az}{\triangleright}

\newcommand{\h}{[\hspace{-0.5pt}[h]\hspace{-0.5pt}]}

\newcommand{\CP}{\mathbb{CP}}
\newcommand{\T}{\mathbb{T}}
\newcommand{\arxiv}{arXiv:}

\begin{document}
\setlength{\leftmargini}{2em} 
\setlength{\leftmarginii}{2em} 

\title[Twist star products and Morita equivalence]{\vspace*{-1.2cm}Twist star products and Morita equivalence \\[12pt]
\textit{Produit {\'e}toile d{\'e}form{\'e} et {\'e}quivalence de Morita}}

\author[F.~D'Andrea]{Francesco D'Andrea}
\address[F.~D'Andrea]{Universit\`a di Napoli ``Federico II'' and I.N.F.N. Sezione di Napoli, Complesso MSA, Via Cintia, 80126 Napoli, Italy.}
\email{francesco.dandrea@unina.it}

\author[T.~Weber]{Thomas Weber}
\address[T.~Weber]{Universit\`a di Napoli ``Federico II'' and I.N.F.N. Sezione di Napoli, Complesso MSA, Via Cintia, 80126 Napoli, Italy.}
\email{thomas.weber@unina.it}


\renewcommand{\keywordsname}{Keywords}

\begin{abstract}
We present a simple no-go theorem for the existence of a deformation quantization of a homogeneous space $M$ induced by a Drinfel'd twist: we argue that equivariant line bundles on $M$ with non-trivial Chern class and symplectic twist star products cannot both exist on the same manifold $M$.
This implies, for example, that there is no symplectic star product on the projective space $\CP^{n-1}$ induced by a twist based on $U(\mathfrak{gl}_n(\C))\h$ or any sub-bialgebra, for every $n\geq 2$.

\bigskip

\noindent\textsc{R{\'e}sum{\'e}.}
Nous exposons un th{\'e}or{\`e}me de non-existence concernant la quantification par d{\'e}formation d'un espace homog{\`e}ne $M$, induite par un twist de Drinfeld: nous montrons qu'un fibr{\'e} en ligne  {\'e}quivariant sur $M$ avec une classe de Chern non-triviale et un produit {\'e}toile symplectique ne peuvent coexister sur une m{\^e}me vari{\'e}t{\'e} $M$. Ceci implique, par exemple, qu'il n'y a pas de produit {\'e}toile  symplectique sur l'espace projectif complexe induit par un twist bas{\'e} sur $U(\mathfrak{gl}_n(\C))\h$, ou sur toute sous-alg{\'e}bre, pour tout $n\geq 2$.
\end{abstract}

\maketitle

\vspace*{-3mm}

\section{Introduction}
Drinfel'd twists \cite{Dri90} are powerful functorial tools to simultaneousely deform bialgebras together with all of their modules and module algebras. Given an action of a Lie algebra $\mathfrak{g}$ on a smooth manifold $M$ by derivations, one can use a twist based on $U(\mathfrak{g})\h$ to obtain a formal deformation quantization of 
$M$ (Def.~\ref{def:1}). A star product obtained by a twist will be called a \emph{twist} star product.

The idea of quantization induced by symmetries has always been appealing in mathematical physics. The approach to quantization via Drinfel'd twist was popularized by several mathematical physicists, among which
Aschieri, Dimitrijevic, Fiore, Lizzi, Meyer, Vitale, Wess and many others
(see e.g.~\cite{ADMW05,ALV07,Asc09,Asc12,Fio10} and references therein),
and its interest is testified by the number of papers on the subject.

The existence of star products for arbitrary Poisson manifolds is a celebrated
result by Kontsevich \cite{Kon03}, that improved previous results by DeWilde and Lecomte \cite{DWL83}, Omori, Maeda and Yoshioka \cite{OMY91} and Fedosov \cite{Fed96}, who proved the existence in the symplectic case (for an historical account see e.g.~the recent review \cite{Wal15}). While every Poisson manifold admits a deformation quantization, it is not clear when a manifold admits a deformation quantization by Drinfel'd twist.
If $M$ is a compact and connected symplectic manifold, we know that it admits a deformation quantization by a twist based on $U(\mathfrak{g})\h$ only if it is a homogeneous space \cite[Thm.~1]{BEWW16}. This condition unfortunately is only necessary, not sufficient: the symplectic $2$-sphere is an example of homogeneous space that admits no deformation quantization induced by a twist \cite[Cor.~3.12]{BEWW16}.

In this short note we comment on how formal Morita equivalence can be used to prove a ``no-go'' theorem for the existence of twist star products. We will prove that:

\begin{nnthm}
Let $G$ be a Lie group, $\mathfrak{g}$ its Lie algebra, $M$ a homogeneous $G$-space and $\omega$ a symplectic form on $M$. The following two properties are mutually exclusive:
\begin{itemize}\itemsep=2pt
\item[(i)] there exists a $G$-equivariant smooth complex line bundle on $M$ with non-trivial Chern class;
\item[(ii)] there exists a deformation quantization of $(M,\omega)$ induced by a twist based on $U(\mathfrak{g})\h$.
\end{itemize}
\end{nnthm}

Using this theorem one can show, for example, that there is no deformation quantization of a symplectic projective space $\CP^{n-1}$ induced by a twist based on $U(\mathfrak{gl}_n(\C))\h$ or any sub-bialgebra (for every $n\geq 2$).
The symplectic $2$-torus on the other hand provides an example where symplectic twist star products exist, and non-trivial $\R^2$-equivariant line bundles do not.

\smallskip

It is worth noticing that if one works in the more general setting of bialgebroids, then any deformation quantization of any Poisson manifold (not only symplectic) is induced by a twist \cite{PXu99}: one can indeed intepret the formal Poisson bivector as a cocycle twist based on the topological bialgebroid of formal power series of differential operators on $M$ (with base algebra $C^\infty(M)\h$).

\subsection*{Notations}
In the following, algebras will be either over the field $\C$ of complex numbers or over the ring $\C\h$ of formal power series in $h$ with complex coefficients.
They are always assumed to be unital and associative.
If $V$ is a complex vector space, we denote by $V\h$ the $\C\h$-module of formal power series in $h$ with coefficients in $V$.

\section{Deformation quantization and Morita equivalence}
Let us recall some basic definitions.

\begin{df}[Star product]\label{def:1}
A \emph{star product} on a Poisson manifold $M$ is a $\C\h$-bilinear associative binary operation $\ast$ on $C^\infty(M)\h$ of the form:
$$
f\ast g=\sum_{k=0}^\infty h^k\, B_k(f,g)\;,\qquad\forall\;f,g\in C^\infty(M),
$$
where each $B_k:C^\infty(M)\times C^\infty(M)\to C^\infty(M)$ is a bi-differential operator, $B_0(f,g)=fg$ is the pointwise multiplication,
$$
B_1(f,g)-B_1(g,f)=\mathrm{i}\left\{f,g\right\}
$$
is the Poisson bracket, and the constant function $1$ is the neutral element of this product:
$$
f\ast 1=1\ast f=f\;,\qquad\forall\;f\in C^\infty(M).
$$
The algebra $(C^\infty(M)\h,\ast)$ is called a \emph{deformation quantization} of the Poisson manifold $M$.
\end{df}

If $M$ is a symplectic manifold, we call $\ast$ a \emph{symplectic} star product (for short). From now on, we will be interested in the symplectic case.

\begin{df}[Equivalence]\label{def:2}
Two star products $\ast$ and $\ast^\prime$ on the same symplectic manifold $M$ are \emph{equivalent} if there are linear maps $T_k:C^\infty(M)\to C^\infty(M)$ such that $T:=\id+\sum_{k=1}^\infty h^kT_k$ satisfies
$T(1)=1$ and:
$$
T(f\ast g)=T(f)\ast^\prime T(g)\;,\qquad\forall\;f,g\in C^\infty(M).
$$
\end{df}

The maps $T_k$ in Def.~\ref{def:2} are automatically differential operators \cite[Thm.~2.22]{GR99}.

The map $T$ gives an isomorphism of algebras $(C^\infty(M)\h,\ast)\to
(C^\infty(M)\h,\ast^\prime)$ by extension of scalars, but note that not every isomorphism is of this form. Any isomorphism that is continuous in the $h$-adic topology can be written as a combination of a change of parameter, an equivalence $T$ as above, and the pullback of a symplectomorphism
\cite[Prop.~9.4]{GR99}.

\smallskip

A weaker notion is that of Morita equivalence of star products, inspired by the notion of (algebraic) Morita equivalence.

\begin{df}[Morita equivalence]
Given two rings (resp.~two algebras) $A$ and $B$, a Morita equivalence $A$-$B$ bimodule is a finitely generated projective right $B$-module $N$ equipped with a ring (resp.~algebra) isomorphism $\phi:A\to\mathrm{End}_B(N)$.
If such a module exists, we say that $A$ and $B$ are \emph{Morita equivalent}.
\end{df}

One can see e.g.~\cite[\S 18]{Lam99} for alternative equivalent definitions of Morita equivalence.

\smallskip

Let $L\to M$ be a smooth complex line bundle, $M$ a symplectic manifold, and $\ast$ a star product on $M$. The space $\Gamma^\infty(L)$ of smooth sections is a symmetric $C^\infty(M)$-bimodule, with left and right module structure given by pointwise multiplication, and by extension of scalars $\Gamma^\infty(L)\h$ is a $C^\infty(M)\h$-bimodule. It was proved in \cite{BW00} (in the more general setting of formal deformations of $\C$-algebras and projective modules) that $\Gamma^\infty(L)\h$ can be deformed into a right module for the algebra 
$(C^\infty(M)\h,\ast)$, i.e.~there is a $\C\h$-bilinear map:
$$
\bullet :\Gamma^\infty(L)\h\times C^\infty(M)\h\to \Gamma^\infty(L)\h \;,
$$
unique modulo equivalences, such that
$$
(s\bullet f)\bullet g=s\bullet (f\ast g)
\;,\qquad
s\bullet 1=s
\;,\qquad
s\bullet f=sf\mod h \;,
$$
 for all $f,g\in C^\infty(M)$, $s\in\Gamma^\infty(L)$. The line bundle determines a second star product $\ast^\prime$ on $M$, unique modulo equivalences, such that there is an isomorphism of $\C\h$-algebras:
\begin{equation}\label{eq:phih}
\phi :(C^\infty(M)\h,\ast^\prime)\to \mathrm{End}_{(C^\infty(M)\h,\ast)}(\Gamma^\infty(L)\h,\bullet ) \;,
\end{equation}
where the latter is the set of all right module endomorphisms, with product given by composition. The isomorphism can be chosen in such a way that it deforms the action of functions on sections by pointwise multiplication \cite[\S 4]{Bur02}:
$$
\phi (f)s=fs\mod h \;,\qquad\forall\;
f\in C^\infty(M), \;s\in\Gamma^\infty(L).
$$
We stress that $\ast$ and $\ast^\prime$ are deformation quantizations of the same symplectic structure on $M$ \cite[Lemma 3.4]{Bur02}. One can prove that the deformed right module above is projective and finitely generated
\cite{BW00}, so that $(C^\infty(M)\h,\ast)$ and $(C^\infty(M)\h,\ast^\prime)$ are Morita equivalent in the ring-theoretic sense.

By \cite[Thm.~3.1]{BW02} the relative class $t(\ast,\ast^\prime)$ of the two star products and the Chern class $c_1(L)$ of the line bundle are proportional: $t(\ast,\ast^\prime)=2\pi\mathrm{i}c_1(L)$. As a consequence:

\begin{lemma}\label{lemma:2}
The star products $\ast$ and $\ast^\prime$ are equivalent if and only if $c_1(L)=0$.
\end{lemma}

\section{Cocycle twists}\label{sec:3}
We refer to \cite{Maj95} for general definitions about bialgebras and module algebras.

\begin{df}[Equivariant module]\label{def:3}
Let $U$ be a $\C$-bialgebra and $A$ and $B$ two $U$-module algebras. An $A$-$B$ bimodule $N$ is called $U$-\emph{equivariant} if it is equipped with a left action $\az$ of $U$ such that:
\begin{equation}\label{eq:emp}
x\az (a\xi b)=(x_{(1)}\az a) (x_{(2)}\az \xi) (x_{(3)}\az b) \;,\qquad\forall\;x\in U,a\in A,b\in B,\xi\in N.
\end{equation}
\end{df}

We will use the same symbol $\az$ for the actions of $U$ on $A$, $B$ and $N$, when there is no risk of confusion; we will also use the standard Sweedler notation for the coproduct, for example above $x_{(1)}\otimes x_{(2)}\otimes x_{(3)}$ stands for $(\Delta\otimes\id)\Delta(x)$.

It was shown in a seminal paper by Drinfel'd \cite{Dri90} that one can modify the coproduct of $U$ by conjugation with an invertible $2$-tensor $F\in U\otimes U$, thus getting a quasi-bialgebra. This is an ordinary bialgebra if the coassociator of $F$ commutes with the image of the iterated coproduct. A special case, recalled below, is obtained when the coassociator is trivial. In this case, we talk about \emph{cocycle twist}, since the defining condition can be interpreted as the vanishing of the coboundary of $F$ in a suitable non-abelian cohomology associated to the bialgebra $U$ (see e.g.~\cite[\S2.3-2.4]{Dan15} and references therein).

\begin{df}[Cocycle twist]
Let $U$ be a $\C$-bialgebra.
An invertible element $F\in U\otimes U$ is called a \emph{cocycle twist} (or simply a \emph{twist}) \emph{based on $U$}, if it satisfies:
\begin{subequations}
\begin{align}
\hspace*{2.5cm} (F\otimes 1)(\Delta\otimes\id)(F) &=
(1\otimes F)(\id\otimes\Delta)(F) \hspace*{1.3cm}
 & \text{(cocycle condition)} \label{eq:3.1a} \\[1pt]
(\varepsilon\otimes\id)(F) &=(\id\otimes\varepsilon)(F)=1 &\text{(counitality)}
\label{eq:3.1b}
\end{align}
\end{subequations}
where $\Delta$ is the coproduct and $\varepsilon$ the counit of $U$.
\end{df}

\begin{rem}\label{rem:7}
We adopt the convention of \cite{CP94,Maj95}. A different convention is that in \cite{ES02,GZ94}, where \eqref{eq:3.1a} is replaced by the condition $(\Delta\otimes\id)(J)(J\otimes 1)=(\id\otimes\Delta)(J)(1\otimes J)$, which is satisfied by the inverse $J=F^{-1}$ of any $2$-cocycle $F$.
\end{rem}

Cocycle twists have the advantage, over more general Drinfel'd twists, that they can be used to obtain \emph{associative} deformations of $U$-module algebras.

Given a bialgebra $U$ and a twist $F$ based on $U$, we denote by $U_F$ the new bialgebra which is given by $U$ as an algebra, with the same counit, and with coproduct $\Delta_F$ given by
$$
\Delta_F(x):=F\Delta(x)F^{-1} \;,\qquad\forall\;x\in U.
$$

Let now $A$ and $B$ be two $U$-module algebras and $N$ a $U$-equivariant $A$-$B$ bimodule like in Def.~\ref{def:3}. Denote by $m_A:A\otimes A\to A$ the multiplication map of $A$, $\lambda_A:A\otimes N\to N$ the left $A$-module action, by $m_B$ the multiplication map of $B$ and $\rho_B:N\otimes B\to N$ the right $B$-module action.

Let $A_F$ be the algebra given by $A$ as a vector space, with the same unit element and with product:
$$
m_{A_F}:=m_A\circ F^{-1} \;.
$$
That is $m_{A_F}(a,b)=m_A\big(F^{-1}(\az\otimes\az)(a\otimes b)\big)$ for all $a,b\in A$.
Similarly let $B_F$ be the algebra given by $B$ as a vector space, with the same unit and with product $m_{B_F}:=m_B\circ F^{-1}$. Both $A_F$ and $B_F$ are $U_F$-module algebras (see e.g.~\cite{Maj95}), w.r.t.~the undeformed action $\az$.

\begin{lemma}\label{lemma:6}
Let $N$ be a $U$-equivariant $A$-$B$ bimodule as in Def.~\ref{def:3}. Then $N$ is a $U_F$-equivariant $A_F$-$B_F$-bimodule w.r.t.~the actions:
\begin{align*}
\lambda_{A_F} &:A_F\otimes N\to N \;, &
\lambda_{A_F}(a\otimes\xi):=
\lambda_A\big(F^{-1}(\az\otimes\az)(a\otimes\xi)\big) \;,\\[1pt]
\rho_{B_F} &: N\otimes B_F\to N \;; &
\rho_{B_F}(\xi\otimes b):=
\rho_B\big(F^{-1}(\az\otimes\az)(\xi\otimes b)\big) \;,
\end{align*}
for all $a\in A,b\in B,\xi\in N$.
\end{lemma}

\begin{proof}
Note that
$$
\lambda_{A_F}(\id\otimes\lambda_{A_F})=
\lambda_A\circ F^{-1}\circ (\id\otimes\lambda_A)\circ(\id\otimes F^{-1})
$$
where we think of $F^{-1}$ as a linear map on tensors, and the action symbol $\az$ is omitted (in this notation $1\,\az$ becomes the identity endomorphism $\id$).
Equivariance of the module means that $F^{-1}\circ (\id\otimes\lambda_A)=(\id\otimes\lambda_A)\circ (\id\otimes\Delta)(F^{-1})$. Thus
\begin{multline*}
\lambda_{A_F}(\id\otimes\lambda_{A_F})=
\lambda_A\circ (\id\otimes\lambda_A)\circ(\id\otimes\Delta)(F^{-1})\circ(\id\otimes F^{-1}) \\
=\lambda_A\circ (m_A\otimes\id)\circ(\Delta\otimes\id)(F^{-1})\circ(F^{-1}\otimes \id) \\
=\lambda_A\circ F^{-1}\circ (m_A\otimes\id)\circ(F^{-1}\otimes \id)
=\lambda_{A_F}(m_{A_F}\otimes\id)
\end{multline*}
where we used the fact that $\lambda_A$ is a left action, the cocycle property of $F$, and the module algebra property telling us that $(m_A\otimes\id)\circ(\Delta\otimes\id)(F^{-1})=F^{-1}\circ (m_A\otimes\id)$. 

From the property $x\az 1_A=\varepsilon(x)1_A\;\forall\;x\in U$ and counitality of $F$, we deduce:
$$
\lambda_{A_F}(1\otimes\xi)=\lambda_A\big((\varepsilon\otimes \id)(F^{-1})(\az\otimes\az)(1\otimes\xi)\big)=\lambda_A(1\otimes\xi)=\xi
$$
for all $\xi\in N$.
The latter two equations prove that $\lambda_{A_F}$ is a left action.
Similarly one proves that $\rho_{B_F}$ is a right action, and that these left and right actions commute, that is
$$
\lambda_{A_F}(\id\otimes\rho_{B_F})=\rho_{B_F}(\lambda_{A_F}\otimes\id)
$$
as linear maps $A\otimes N\otimes B\to N$. Finally, thinking of $x\in U$ as a linear map and omitting the action symbol $\az$, one finds:
\begin{align*}
x\circ\lambda_{A_F}&\circ (\id\otimes\rho_{B_F}) =
x\circ\lambda_A\circ F^{-1}\circ(\id\otimes\rho_B)\circ(\id\otimes F^{-1})
\\
&=\lambda_A\circ\Delta(x)\circ F^{-1}\circ(\id\otimes\rho_B)\circ(\id\otimes F^{-1}) \\
&=\lambda_A\circ F^{-1}\circ \big(F\Delta(x)F^{-1}\big)\circ(\id\otimes\rho_B)\circ(\id\otimes F^{-1}) \\
&=\lambda_A\circ F^{-1}\circ (\id\otimes\rho_B)\circ
\big((\id\otimes\Delta)(F)\,(\id\otimes\Delta)\,\Delta(x)(\id\otimes\Delta)(F^{-1})\big)
\circ
(\id\otimes F^{-1}) \\
&=\lambda_{A_F}\circ (\id\otimes\rho_{B_F})\circ
(\id\otimes\Delta_F)\Delta_F(x) \;,
\end{align*}
where we used the $U$-equivariance of $\lambda_A$ and $\rho_B$.
This proves the $U_F$-equivariance of the actions $\lambda_{A_F}$ and $\rho_{B_F}$.
\end{proof}

We will denote by $N_F$ the bimodule given by the vector space $N$ with actions $\lambda_{A_F}$ and $\rho_{B_F}$ given in Lemma \ref{lemma:6}.
Analogous definitions and constructions work for topological bialgebras over the ring $\C\h$, with algebraic tensor products replaced by tensor products completed in the $h$-adic topology.

\section{Twist star products}

Let $G$ be a Lie group, $\mathfrak{g}$ its Lie algebra, $\pi:L\to M$ a $G$-equivariant line bundle over a real smooth manifold $M$, i.e.~both $L$ and $M$ are $G$-spaces, the action of $G$ commutes with $\pi$, and is linear on fibers.
Define an action $\alpha$ of $G$ on $f\in C^\infty(M)$ and $s\in \Gamma^\infty(L)$ by
$$
\alpha_gf(x):=f(g^{-1}x) \;;\qquad
\alpha_gs(x):=g\,s(g^{-1}x) \;,
$$
for all $x\in M$ and $g\in G$.
The equivariance condition of $\pi$ guarantees that $\alpha_gs$ is still a section of $L$, indeed:
$$
\pi\big(\alpha_gs(x)\big)=
\pi\big(gs(g^{-1}x)\big)=
g\pi\big(s(g^{-1}x)\big)=g(g^{-1}x)=x \;,
$$
that means $\pi\circ\alpha_gs=\id_M$.
Note that $\alpha_g(fs)=\alpha_g(f)\alpha_g(s)$ (equivariance condition for grouplike elements of a bialgebra).

By differentiating this action we get an action of the bialgebra $U(\mathfrak{g})\h$ on $C^\infty(M)\h$ and on $\Gamma^\infty(L)\h$ that turns the latter into a $U(\mathfrak{g})\h$-equivariant $C^\infty(M)\h$-bimodule. It is a Morita equivalence bimodule, with isomorphism
$$
\psi_0:C^\infty(M)\h\to \mathrm{End}_{C^\infty(M)\h}(\Gamma^\infty(L)\h)
$$
given by pointwise multiplication: $\psi_0(f)s:=fs\;\forall\;f\in C^\infty(M),s\in \Gamma^\infty(L)$.

\smallskip

Given a cocycle twist
$$
F=\sum_{k=0}^\infty h^kF_k
$$
based on $U(\mathfrak{g})\h$ (here $F_k\in U(\mathfrak{g})\otimes U(\mathfrak{g})$ for all $k\geq 0$), we can now apply the prescription in \S\ref{sec:3} and get a deformed multiplication on $A:=C^\infty(M)\h$ (note that here, in the notations of \S\ref{sec:3}, we have $A=B$) and a deformed bimodule structure on $N:=\Gamma^\infty(L)\h$.

It is worth noticing that $U(\mathfrak{g})\h_F$ is a \emph{deformation} of the bialgebra $U(\mathfrak{g})\h$ (in the sense e.g.~of \cite[Def.~6.1.1]{CP94}) \,---\, that is
$$
\Delta_F(x)=\Delta(x)\mod h
$$
for all $x\in U(\mathfrak{g})$ \,---\,
if and only if $F_0$ commutes with the image of $\Delta$.

As one can easily check, this implies that $\widetilde{F}:=FF_0^{-1}$ is a cocycle twists, and $\Delta_{\widetilde{F}}\equiv\Delta_F$. Modulo a replacement of $F$ by $\widetilde{F}$, we can then assume that our twist is of the form
\begin{equation}\label{eq:oftheform}
F=1\otimes 1 \mod h.
\end{equation}
A byproduct of condition \eqref{eq:oftheform} is that $f\ast g:=m_A\circ F^{-1}(\az\otimes\az)(f\otimes g)$ is equal to $fg\mod h$, i.e.~a star product according to Def.~\ref{def:1}.

As customary, we will include \eqref{eq:oftheform} in the definition of \emph{formal} twist (see e.g.~\cite[\S9.5]{ES02}).

\begin{df}[Twist star product]
A star product of the form
\begin{equation}\label{eq:tsp}
f\ast g:=m\circ F^{-1}(\az\otimes\az)(f\otimes g) \qquad\;\forall\;f,g\in C^\infty(M)\h \;,
\end{equation}
with $m$ the pointwise multiplication of $C^\infty(M)\h$ and $F$ a twist satisfying \eqref{eq:oftheform}, will be called \emph{twist star product}.
\end{df}

\begin{prop}\label{prop:8}
Let $\lambda_{A_F}$ and $\rho_{A_F}$ be the module actions of $A_F=(C^\infty(M)\h,\ast)$ on the set $N=\Gamma^\infty(L)\h$ given in Lemma \ref{lemma:6}. The map
\begin{equation}\label{eq:psih}
\psi :A_F\to\mathrm{End}_{A_F}(N_F)
\end{equation}
from $A_F$ into the algebra of right $A_F$-linear endomorphisms given by
$$
\psi (f)s:=\lambda_{A_F}(f\otimes s) \;,\qquad\forall\;f\in C^\infty(M), s\in \Gamma^\infty(L),
$$
is an algebra isomorphism.
\end{prop}

\begin{proof}
Since $\psi (f)s=fs\mod h$, one has $\psi =\psi_0+O(h)$. Since $\psi_0$ is an invertible map, $\psi $ is invertible as well (a formal power series is invertible if{}f its order zero term is invertible). Note that $\lambda_{A_F}(1\otimes s)=s$, so that $\psi $ is an isomorphism of \emph{unital} algebras.
\end{proof}

In the terminology of \cite[Def.~4.2]{Bur02}, $N_F$ is a \emph{bimodule quantization} of the line bundle $L$.
We can now prove our main theorem.
The technique employed is similar to that used in \cite[Cor.~6.7]{BNWW07} to prove the non-existence of formal deformations of ``sufficiently non-trivial'' principal bundles.

\begin{thm}\label{thm:9}
Let $G$ be a Lie group, $\mathfrak{g}$ its Lie algebra, $M$ a homogeneous $G$-space and $\omega$ a symplectic form on $M$. The following two properties are mutually exclusive:
\begin{itemize}\itemsep=2pt
\item[(i)] there exists a $G$-equivariant smooth complex line bundle on $M$ with non-trivial Chern class;
\item[(ii)] there exists a deformation quantization of $(M,\omega)$ induced by a twist based on $U(\mathfrak{g})\h$.
\end{itemize}
\end{thm}

\begin{proof}
Let us assume that (i) and (ii) both hold, and show that we arrive at a contraddiction. In the notations above, let $F$ be a twist, $\ast$ a symplectic star product induced by $F$, $L\to M$ an equivariant line bundle with $c_1(L)\neq 0$, $A_F$ and $N_F$ like in Prop.~\ref{prop:8}. The line bundle $L$ induces a second star product $\ast^\prime$ on $M$. On the other hand, composing
\eqref{eq:phih} with the inverse of \eqref{eq:psih} we get an isomorphism (of unital algebras):
$$
T:=\psi ^{-1}\circ\phi :(C^\infty(M)\h,\ast^\prime)\to (C^\infty(M)\h,\ast)
$$
that satisfies $T(f)=f\mod h\;\forall\;f\in C^\infty(M)$.
It is then an equivalence between $\ast$ and $\ast^\prime$, in contraddiction with Lemma \ref{lemma:2} that states that $\ast$ and $\ast^\prime$ are not equivalent.
\end{proof}

\section{Applications}

\subsection{Complex projective spaces}
As a first example, let us consider the complex projective space $\CP^{n-1}$, $n\geq 2$. The tautological line bundle has total space:
$$
L:=\big\{
(\ell,v) \in \CP^{n-1}\times\C^n: v\in\ell
\big\}
$$
where points $\ell\in\CP^{n-1}$ are lines through the origin in $\C^n$. The bundle map is simply $\pi:L\to\CP^{n-1}$, $\pi(\ell,v)=\ell$.

The action of $GL_n(\C)$ on $\C^n$ by row-by-column multiplication induces an action on $\CP^{n-1}$ (it sends $1$-dimensional vector subspaces of $\C^n$ into $1$-dimensional vector subspaces); the diagonal action on $\CP^{n-1}\times\C^n$ induces an action on $L$ commuting with $\pi$. It is then an equivariant line bundle. Since $c_1(L)\neq 0$, as a corollary of Theorem \ref{thm:9}:

\begin{cor}\label{cor:12}
There is no symplectic star product on $\CP^{n-1}$ induced by a twist based on\linebreak $U(\mathfrak{gl}_n(\C))\h$ or any sub-bialgebra.
\end{cor}

Fuzzy spaces belong to this class of examples. From a mathematical point of view, fuzzy spaces are strict deformation quantizations of coadjoint orbits of connected compact semisimple Lie groups, obtained via covariant Berezin quantization (see e.g.~\cite{Rie04}). Alternatively, since on any such orbit there is a canonical invariant K{\"a}hler structure (see e.g.~\cite{Rie09}), they can also be obtained via Berezin-Toepliz quantization \cite{Sch09}. It was shown by Schlichenmaier in \cite{Sch00} (see \cite{Sch96} for the original reference in German), using some estimates of \cite{BMS93}, that one can associate a natural star product to the Berezin-Toepliz quantization of any compact K{\"a}hler (hence symplectic) manifold, such as $\CP^{n-1}$. Corollary \ref{cor:12} can be applied to such star products on $\CP^{n-1}$ to conclude that they are not induced by a twist based on $U(\mathfrak{gl}_n(\C))\h$.


\subsection{The noncommutative 2-torus}
A prototypical example of symplectic twist star product is the Moyal-Weyl product on $\R^{2n}$, or its compact version: the (formal) noncommutative torus. Let $\T^2:=\R^2/\Z^2$, and denote by $x,y$ the Cartesian coordinates on $\R^2$. A global frame for vector fields on $\T^2$ is given by the partial derivatives $\partial_x$ and $\partial_y$. The Lie algebra generated by such derivations will be identified with $\R^2$.

Weyl's star product on $\T^2$ can be written in the form \eqref{eq:tsp}, with
\begin{equation}\label{eq:MoyalWeyl}
F:=\exp \frac{ih}{2}\big\{ \partial_y\otimes\partial_x-\partial_x\otimes\partial_y \big\}
\end{equation}
a twist based on $U(\R^2)\h$. Together with the twists in \cite{BTY05,GZ94} based on the $ax+b$ group, and to the construction in \cite{ESW16} relying on Fedosov techniques, \eqref{eq:MoyalWeyl} is one of the few instances of twist that can be written down explicitly.

Using \eqref{eq:MoyalWeyl} one gets
a deformation quantization of $\T^2$ w.r.t.~its standard symplectic structure.
As a consequence of Theorem \ref{thm:9}:

\begin{cor}
Every $\R^2$-equivariant smooth complex line bundle $L$ on $\T^2$ has $c_1(L)=0$.
\end{cor}

Of course, it is not difficult to give a direct proof (not relying on Theorem \ref{thm:9}) of this simple fact. Suppose $L\to\T^2$ is $\R^2$-equivariant. Denote by
$$
\az:\R^2\times\Gamma^\infty(L)\to \Gamma^\infty(L)
$$
the corresponding action of the Lie algebra $\R^2$ on the module of sections. Then, the formula
$$
\nabla_{a\partial_x+b\partial_y}s:=a(\partial_x\az s)+b(\partial_y\az s) \;,\qquad\forall\;a,b\in C^\infty(\T^2),s\in\Gamma^\infty(L),
$$
defines a \emph{flat} connection $\nabla$ on $L$. Indeed, property \eqref{eq:emp} guarantees that $\nabla$ satisfies the Leibniz rule:
$$
\nabla_X(fs)=X(f)s+f\nabla_X(s) \;,\qquad\forall\;X=a\partial_x+b\partial_y\in\mathfrak{X}(\T^2), f\in C^\infty(\T^2),
$$
and clearly the connection $1$-form of $\nabla$ is zero, that means
$c_1(L)=0$.

\section*{Acknowledgement}
We thank Chiara Esposito and Stefan Waldman for many useful discussions and for their comments on a first version of the paper.
This paper was written while F.D.\ was visiting for a semester Penn State University: F.D.\ is grateful to the institution for the excellent working conditions, and to Ping Xu for his constant support and help.

\end{document}